\documentclass[reqno]{amsart}
\usepackage{epsf,epsfig}
\usepackage{graphics,color}
\usepackage{amsmath}
\usepackage{amssymb}
\usepackage{cite}
\usepackage{verbatim}
\usepackage{float}
\usepackage{graphicx}
\usepackage{amsthm}
\usepackage{textcomp}
\usepackage{subfig}
\usepackage{pgf}

\newtheorem{theorem}{Theorem}[section]
\theoremstyle{definition}

\numberwithin{equation}{section}
\renewcommand{\d}{\mathrm{d}}
\newcommand{\D}{\mathrm{D}}

\newcommand{\Rz}{{\mathbb R}}

\newcommand{\disp}{\displaystyle}
\newcommand{\haz}{\widehat}
\newcommand{\ove}{\overline}

\newcommand{\weakto}{\rightharpoonup}

\newcommand{\lan}{\langle}
\newcommand{\ran}{\rangle}

\begin{document}

\title[Minimizing movements for GENERIC]{A minimizing-movements
  approach \\ to GENERIC systems}
\author[A. J\"ungel]{Ansgar J\"ungel}
\address[Ansgar J\"ungel]{Institute for Analysis and Scientific Computing, Vienna University of Technology, Wiedner
Hauptstra\ss e 8-10, 1040 Wien, Austria}
\email{juengel@tuwien.ac.at}
\urladdr{http://www.asc.tuwien.ac.at/$\sim$juengel}
\author[U. Stefanelli]{Ulisse Stefanelli} 
\address[Ulisse Stefanelli]{Faculty of Mathematics, University of
  Vienna, Oskar-Morgenstern-Platz 1, A-1090 Vienna, Austria,
Vienna Research Platform on Accelerating
  Photoreaction Discovery, University of Vienna, W\"ahringerstra\ss e 17, 1090 Wien, Austria,
 \& Istituto di
  Matematica Applicata e Tecnologie Informatiche {\it E. Magenes}, via
  Ferrata 1, I-27100 Pavia, Italy
}
\email{ulisse.stefanelli@univie.ac.at}
\urladdr{http://www.mat.univie.ac.at/$\sim$stefanelli}
\author[L. Trussardi]{Lara Trussardi}
\address[Lara Trussardi]{Faculty of Mathematics, University of Vienna, 
Oskar-Morgenstern-Platz 1, 1090 Wien, Austria.}
\email{lara.trussardi@univie.ac.at}
\urladdr{http://www.mat.univie.ac.at/$\sim$trussardi}

\date{\today}


 \begin{abstract}
We present a new time discretization scheme adapted to the
structure of GENERIC systems. The scheme is variational in nature and is based on a
{\it conditional} incremental minimization. The GENERIC structure of
the scheme provides stability and convergence of the scheme. 
We prove that the scheme can
be rigorously implemented in the case of the damped harmonic oscillator. Numerical
evidence is collected, illustrating the performance of the method.
 \end{abstract}

\maketitle

%%%%%%%%%%%%%%%%%%%%%%%%%%%%%%%%%%%%%%%%%%%%%%%%%%%%%%%%%%%%%%%%%%%%%%%%%%%

\section{Introduction} 

The aim of this note is to discuss a new variational time-discretization scheme
adapted to the structure of General Equations for Non-Equilibrium
Reversible-Irreversible Coupling (GENERIC). Introduced by Grmela \& \"Ottinger,
this formulation provides a unified frame for describing the time evolution of physical systems out of equilibrium in
presence of reversible and irreversible dynamics \cite{Oettinger}.

Let $y$ denote the {\it state} of a closed, nonequilibrium physical system
and let $E(y)$ and $S(y)$ be
the corresponding total {\it energy} and total {\it entropy}, respectively. The GENERIC formulation of the time
evolution of the system reads
\begin{equation}
  \label{gen0}
  y' = L(y) \,\D E(y) +K(y)  \, \D  S (y).
\end{equation}
Here, $y'$ denotes the time derivative, $L$ is the antisymmetric {\it
  Poisson} operator, and $K$ is the symmetric and positive definite
{\it Onsager} operator (details in Section \ref{sec:gen}). 

The gist of the GENERIC system \eqref{gen0} is that conservative and
dissipative dynamics are clearly separated. Under a specific compatibility condition, see
\eqref{nic} below, this entails that a solution $t\mapsto y(t) $ to
\eqref{gen0} conserves energy and accumulates entropy in a  
quantifiable way, namely
$$ \frac{\d}{\d t} E(y) = 0 \quad \text{and} \quad \frac{\d}{\d t}
S(y) =\lan \D S(y) , K(y) \D S\ran\geq 0.$$
 The conservation of energy and the quantified dissipative
character  are the distinguishing traits of the GENERIC
system \eqref{gen0}.
To replicate these properties at the discrete level leads to
so-called {\it structure-preserving} approximations. These have
drawn interest in the last years, giving rise to different
numerical solutions adapted to different applicative
contexts.

In the last years, GENERIC 
has attracted increasing
attention and has been applied to a number of situations ranging from
complex fluids \cite{Grmela}, to dissipative quantum mechanics \cite{Mielke13}, to
thermomechanics \cite{Betsch,Huetter,Huetter2,Huetter3,Mielke11}, to
electromagnetism \cite{Jelic}, to shape-memory alloys \cite{generic_souza},
to the Vlasov-Fokker-Planck \cite{Duong,Hoyuelos,Peletier} and the
Boltzmann equation \cite{Montefusco}, and to large-deviation limits of reversible stochastic processes \cite{Kraaij,Kraaij2}.

Numerical schemes conserving energy
can be found for instance in
\cite{Gonzales,QuMc08,Romero,Romero3,Romero2,SMSF15}; see
\cite{CGMMOOQ12} for a review and \cite{Suzuki} for a contribution
explicitly focusing on GENERIC formulations. These schemes are often {\it discrete-gradient}
methods, where gradients of functionals are specifically modified in
order to fulfill a discrete chain rule and exactly replicate
conservation \cite{Gonzales,Gonzales2,HLW06,MQR99}. In the thermomechanical
context, structure-preserving discretizations either in terms of the absolute
temperature \cite{Conde,Conde2,Garcia,Portillo} or of the  internal energy
or the entropy \cite{Betsch,Betsch2} have been obtained. The reader is
referred to 
\cite{Krueger} for an approach to open systems. 

GENERIC integrators
able to conserve energy and accurately describe entropy accumulation 
have been proposed in 
\cite{Oettinger2}, where however some limitations are also mentioned. In particular, energy and Onsager
operators have to be suitably modified in a time-step dependent manner in order energy conservation to
hold. Integrators are actually constructed in case of single dissipation
mechanism only (were $K$ be a matrix, it would have to
have rank one) and no convergence theory is provided. The discretization of the
damped harmonic oscillator is addressed in \cite{Oettinger2}, where
nonetheless the temperature is given by a prescribed heat bath. In this case,
explicit solutions have to be used in order to specify the GENERIC integrator.

To the best of our knowledge, all available
numerical schemes directly target GENERIC systems in their
differential form~\eqref{gen0}. Our focus is here on a time-discretization of variational nature instead,
fitting into the general scheme of so-called {\it minimizing
movements} \cite{Ambrosio95,Ambrosio08}. In particular, at all discrete steps we
aim at solving a specific minimization problem. We start by defining
the entropy-production potential
$\Psi(y,\xi) = \lan\xi, K(y)\xi \ran/2$ and denoting by
$\Psi^*(y,\cdot)$ its conjugate in the second variable. Then,  
the GENERIC relation \eqref{gen0} in $[0,T]$ can be equivalently rewritten in
terms of the scalar equation
\begin{equation}-S(y(t)) + \int_0^t \Psi^*(y,y'{-}L(y) \, \D E(y)) + \int_0^t
\Psi(y,\D S(y)) +S(y(0))=0\label{gen01}
\end{equation}
for all $t\in [0,T]$. The
reformulation of dissipative systems in terms of scalar equations
as \eqref{gen01} is usually referred to as {\it De Giorgi's Energy-Dissipation
  Principle} \cite{Dondl,Liero,Mielke20}. This principle has already
been applied in a variety of different contexts, including generalized
gradient flows \cite{Bacho,Rossi08}, curves
of maximal slope in metric spaces \cite{Ambrosio08,DeGiorgi}, rate-independent systems
\cite{Mielke12,Roche}, and optimal control \cite{portinale}. For
GENERIC systems, the reformulation of \eqref{gen0} in terms of such
 variational principle has been already presented in \cite{Peletier}.

Our new minimizing-movements approach to \eqref{gen0} consists in tackling a
discrete version of relation \eqref{gen01}.
Assume to be
given a time partition $0=t_0<t_1 < \dots<t_N=T$ with steps $\tau_i  =
t_i - t_{i-1}$, and an initial state $y^0$. We define the time-discrete trajectory $\{ y_i \}_{i=0}^N$ by letting $y_0=y^0$ and
subsequently perform the minimization  
$$\min_y\left\{ -S(y) +\tau_i \Psi^*\left(y, \frac{y- y_{i-1}}{\tau_i} -
      L(y) \, \D E(y)\right) + \tau_i \Psi(y,\D S(y)) +
    S( y_{i-1})\right\}$$
for $ i=1,\dots, N$. The latter is nothing but a localized and
discretized version of the De Giorgi's Energy-Dissipation Principle
\eqref{gen01}. 

The minimizing-movements scheme above has been
introduced in \cite{generic_euler}. The
theory in \cite{generic_euler} is however restricted to the case of  state-independent
operators $L$ and $K$, which severely limits the applicability to real GENERIC systems. In addition, the convergence analysis in
\cite{generic_euler} relies on a suitable set of a priori assumptions,
leaving open the discussion whether these can be met in practice.

The aim
of this note is then threefold. First, we extend the reach of the
numerical method  to include the case of state dependent operators $L(y)$ and
$K(y)$, hence covering the full extent of the GENERIC theory (Section
\ref{sec:gen}). Our main result is the conditional convergence of
Theorem \ref{thm:cond}. Second, we
provide the detailed analysis of a specific case, the damped
harmonic oscillator, in which the above-mentioned convergence assumptions can
actually be proved to hold (Section \ref{sec:oscillator}). Eventually, we present numerical
experiments assessing the performance of the minimizing-movements scheme and
compare it to the classical
implicit Euler scheme (Subsection \ref{sec:num}).

%%%%%%%%%%%%%%%%%%%%%%%%%%%%%%%%%%%%%%%%%%%%%%%%%%%%%%%%%%%%%%%%%%%%%%%%%%%

\section{The minimizing-movements  scheme for GENERIC
  systems} \label{sec:gen}

In this section, we recall the structure of a GENERIC system \cite{Oettinger} by
specifying the assumptions on functionals and operators that will be
used throughout. Moreover, we formulate our minimizing-movements  scheme
and present a conditional convergence result, namely Theorem \ref{thm:cond}.

The GENERIC system 
\begin{equation}
  \label{gen}
  y' = L(y) \,\D E(y) +K(y)  \, \partial  S (y) \quad \text{a.e. in} \ [0,T]
\end{equation} 
is defined by specifying the quintuple
$(Y,E,S,L,K)$. In the following, the state space $Y$ is assumed to be a reflexive Banach 
space. The functionals $E$ and $S$ represent the total energy and the total entropy of the
system, respectively. We assume $E$ to be Fr\'echet differentiable, with
strongly-weakly continuous differential $\D E$, and $-S:Y \to
(-\infty,\infty]$ to be proper and lower semicontinuous
with single-valued and strongly-weakly continuous Fr\'echet subdifferential $\partial(-S)$. Recall 
that one has $\xi \in \partial(-S)(y)$ iff $ S(y)>-\infty $ and 
$$ \liminf_{x\to y}\frac{S(y)-S(x) - \lan \xi , x-y \ran}{\| x-y
  \|_Y}\geq 0.$$ 
In the following, we also make use of the obvious notation  $-\partial S
= \partial (-S)$.

The operators $L$ and $K $ define a  {\it
  Poisson} and an {\it Onsager} geometric structure on $Y$,
respectively. In particular, for all
states $y\in Y$ we assume that 
$L(y)$ and $K(y)$ are linear and continuous from $Y^*$ to
$Y$. Moreover, $L(y)$ is required to be antisymmetric, $L^*(y)=-L(y)$, and to fulfill the Jacobi
identity $\{ \{g_1,g_2\},g_3 \}+\{ \{g_2,g_3\},g_1 \}+\{
\{g_3,g_1\},g_2 \}=0$. Here, $g_i$ denotes any differentiable function on $Y$
and the Poisson bracket is defined as $\{g,\tilde g\}(y)=\lan \D g(y),
L(y)\D \tilde g(y)\ran $, where $\lan \cdot,\cdot \ran$ denotes the
duality pairing
between $Y^*$ (dual) and $Y$. Moreover, we assume the strong-weak
continuity
\begin{equation}
y_n
\stackrel{Y}{\to} y \ \ \text{and} \ \ \xi_n \stackrel{Y^*}{\weakto}
\xi \ \ \Rightarrow \ \ L(y_n)\xi_n \stackrel{Y}{\weakto} L(y)\xi.\label{eq:contL}
\end{equation} 

The mapping $K(y)$ is asked to be symmetric
and positive definite, namely $ K(y)=K^*(y)\geq 0$. We
associate to $K$ the so-called {\it entropy-production potential} $\Psi: Y \times Y^* \to [0,\infty)$ given
by $\Psi(y,\xi) = \lan \xi, K(y)\xi\ran/2$ and let $\Psi^*$
 be its conjugate in the second variable, namely $\Psi^*(y,\eta) = \sup_{\xi}(\lan \xi,
\eta \ran  - \Psi(y,\xi))$. We also assume the lower semicontinuity of the sum of
the entropy-production potential and its dual,  that is
\begin{align}&
  y_n
\stackrel{Y}{\to} y, \ \  \eta_n \stackrel{Y}{\weakto}
\eta, \ \ \text{and} \ \ \xi_n \stackrel{Y^*}{\weakto}
\xi \nonumber \\
&\quad \Rightarrow \ \ \liminf_{n \to\infty} \left(\Psi^*(y_n,\eta_n)+
 \Psi(y_n,\xi_n)\right) \geq \Psi^*(y,\eta) + \Psi(y,\xi). \label{eq:contK}
\end{align}

In addition, functionals and operators are asked to satisfy the crucial
{\it noninteraction} condition 
\begin{equation}
  \label{nic}
  L^*(y)\, \partial S(y) = K^*(y)\, \D E(y) = 0 \quad \forall y \in Y.
\end{equation}
This condition ensures that the system of dissipative
actions $K(y)\,\partial S(y)$ does not spoil energy conservation and the system of conservative actions $L(y)\,\D E(y)$
does not contribute to dissipation. Indeed, by
assuming sufficient smoothness one checks that 
\begin{align}
\frac{\d}{\d t} E(y) &= \lan \D E(y), y' \ran \stackrel{\eqref{gen}}{
  =}\lan \D E(y), L(y) \D E(y) \ran + \lan \D E(y), K(y) \partial S(y)  \ran
   \nonumber \\
&= 0 +  \lan  \partial S(y), K^*(y)\D E(y)  \ran\stackrel{\eqref{nic}}{
  =} 0, \nonumber \\ %\label{E} \\
\frac{\d}{\d t} S(y) &= \lan \partial S(y), y' \ran \stackrel{\eqref{gen}}{
  =}\lan \partial S(y), L(y) \D E(y) \ran + \lan \partial S(y), K(y) \partial S(y)  \ran \nonumber
  \\
&= \lan \D E(y)  , L^*(y)\partial S(y)\ran + \lan \partial S(y), K(y) \partial S(y)  \ran \nonumber\\
&\stackrel{\eqref{nic}}{
  = } \lan \partial S(y), K(y) \partial S(y)  \ran  \geq 0.\label{S}
\end{align}
The noninteraction condition \eqref{nic} hence implies that trajectories $y$ solving
\eqref{gen} have constant total energy and entropy rate $\lan \partial
S(y), K(y) \partial S(y)  \ran $. In particular, the entropy is
nondecreasing and entropy production results solely from irreversible processes.
In computing \eqref{S}
 we have used the chain rule $(\d/\d t)S(y) = \lan \partial
S(y),y'\ran$ almost everywhere. This is classical in case $-S$ is (a
regular perturbation of) a convex functional  \cite[Prop. 3.3,
p. 73]{Brezis73}. The reader is referred to \cite{Rossi-Savare06} for
a general discussion out of the convex case.

Before moving on, let us remark that the structure of GENERIC is
geometric in nature. Indeed, it is invariant by coordinate changes. Let $y = \phi(\tilde y)$ for
$\tilde y \in \tilde Y$ and define
$\tilde E(\tilde y) =  E(\phi( \tilde y))$, $-\tilde S(\tilde y) =
-S(\phi(\tilde y))$, $\tilde L(\tilde y) = \D \phi(\tilde y)^{-1}
L(\phi(\tilde y)) \D \phi(\tilde y)^{-*}$, and $\tilde K(\tilde y) = \D \phi(\tilde y)^{-1}
K(\phi(\tilde y)) \D \phi(\tilde y)^{-*}$, where $\D \phi(\tilde y) ^{-*}:
\tilde Y^* \to Y^*$ is the adjoint of
the inverse of $\D \phi(\tilde y) :  \tilde Y \to
Y$. Then, the quintuple $(\tilde Y, \tilde E, \tilde
S, \tilde L,\tilde K)$ satisfies the above structural assumptions and
the GENERIC structure \eqref{gen} can be rewritten as $\tilde y' = \tilde L( \tilde y)\, \D \tilde
E(\tilde y) + \tilde K (\tilde y)\, \partial \tilde S(\tilde y)$.

We now reconsider the discussion leading to \eqref{gen01} and specify it
further by remarking that relation \eqref{gen} is actually
equivalent to the {\it inequality}
 \begin{align}&-S(y(t)) + \int_0^t \Psi^*\big(y,y'{-}L(y) \, \D
   E(y)\big) + \int_0^t
\Psi\big(y,\partial S(y)\big)+S(y(0))\leq 0 \label{gen1}
\end{align}
for all $t \in [0,T]$.
The equivalence between \eqref{gen} and \eqref{gen1} follows from
Fenchel's relations. Applied to the entropy-production potential
$\Psi$, these relations  read  
\begin{align}
  \Psi^*(y,\eta) + \Psi(y,\xi) \geq \lan \xi,\eta \ran \quad \forall
  y,\,\eta \in Y, \ \xi \in Y^*,
  \label{eq:fenchel}\\
\Psi^*(y,\eta) + \Psi(y,\xi) = \lan \xi,\eta \ran\quad \Leftrightarrow \quad
\xi\in \partial \Psi(y,\eta),
  \label{eq:fenchel2}
\end{align}
where the subdifferential is taken in the second variable only.
By noting that $\partial \Psi(y,\partial S(y)) = K(y)\, \partial S(y)$ and using \eqref{eq:fenchel}-\eqref{eq:fenchel2}, one can prove the
equivalences 
\begin{align*}
  &\eqref{gen}  \quad \Leftrightarrow   \ \  y' - L(y)\, \D E(y)
  = \partial \Psi(y,\partial S(y)) \ \ \text{a.e.}\quad
  \nonumber\\
&\quad\stackrel{\eqref{eq:fenchel2}}{\Leftrightarrow} 
  \Psi^*(y,y'{-}L(y)\, \D E(y))+ \Psi(y,\partial S(y)) - \lan y'{-}L(y)\, \D
  E(y), \partial S(y)\ran \leq 0 \ \ \text{a.e.}\quad
  \nonumber\\
&\quad\stackrel{\eqref{nic}}{\Leftrightarrow}  \Psi^*(y,y'{-}L(y)\, \D E(y))+ \Psi(y,\partial
  S(y)) - \frac{\d}{\d t} S(y) \leq 0  \ \ \text{a.e.} \nonumber\\
&\quad \stackrel{\eqref{eq:fenchel}}{\Leftrightarrow} \quad \eqref{gen1}.
\end{align*}
In particular, the last left-to-right implication follows by
integration while the right-to-left counterpart from the nonnegativity
of the integrand, given \eqref{nic} and \eqref{eq:fenchel}.

The minimizing-movements scheme corresponds to a discretization of
inequality \eqref{gen1}.
To each time partition $0=t_0<t_1<\dots< t_{N}=T$, we associate the
time steps $\tau_i = t_i - t_{i-1}$ and the diameter $\tau = \max
\tau_i$. Given the vector $\{y_i\}_{i=0}^N\in Y^{N+1}$, we introduce the
backward piecewise
constant and piecewise linear interpolations $\ove y:[0,T] \to Y$ and
$\haz y:[0,T] \to Y$,
\begin{align*}
&\ove y(0)= \haz y(0), \ \ove y(t)=y_i, \ \haz y(t) =
\frac{t-t_{i-t}}{\tau_i}y_{i-1}+ \frac{t_i - t}{\tau_i}y_i,\\
&\qquad
\forall t \in (t_{i-1},t_i], \ i =1, \dots, N.
\end{align*}

We define the {\it incremental} functional by
$G :  (0,\infty) \times Y\times Y \to (-\infty,\infty]$ as
$$G(\tau,\eta; y) = -S(y) +\tau \Psi^*\left(y, \frac{y-\eta}{\tau} -
      L(y) \, \D E(y)\right) + \tau \Psi(y,\partial S(y)) +
    S(\eta).$$
By letting $y_0= y^0$ we find the discrete solution
 $\{y_i\}_{i=0}^{N}$ by subsequently solving the minimization
problem
\begin{equation}\min_y G(\tau_i,y_{i-1};y ) \quad \text{for} \  i =1, \dots,
N.\label{eq:min}
\end{equation}

Note that, for all $\tau>0$ and $y_{i-1}\in Y$ with $S(y_{i-1})>-\infty$,
the map $y \mapsto G(\tau_i,y_{i-1};y)$ is strongly lower
semicontinuous because of the lower semicontinuity of $-S$, the lower
semicontinuity \eqref{eq:contK} of $\Psi^*+\Psi$, the weak-strong continuity of $\D E $ and $\partial
S$, and the continuity \eqref{eq:contL} of $L$. In order to solve
problem \eqref{eq:min} one has hence to check that $y \mapsto G(\tau_i,y_{i-1};y)$ is strongly
coercive.

The main result of this section is the following {\it conditional} convergence result.

\begin{theorem}[Conditional convergence]\label{thm:cond} Under the above assumptions, let a sequence of partitions $0=t_0^n<\dots<t_{N^n}^n=T$ be
  given with $ \tau^n\to 0$ as $n \to \infty$, and let $\{y_i^n\}_{i=1}^{N^n}$ be such that $y_0^n = y^0$. Assume that 
  \begin{align}
    &\sum_{i=1}^{N^n}\big(G(\tau_i^n , y_{i-1}^n; y_i^n)\big)^+ \to 0 \ \
    \text{as} \ n \to \infty, \label{eq:G}\\
& \{\haz y^n\} \ \text{is bounded in} \ H^1(0,T;Y)  \ \text{and
  takes values in} \  K \subset\subset Y,\label{eq:G2}\\
&\{L(\ove y^n)\,\D E(\ove y^n)\} \ \text{is bounded in} \ L^2(0,T;Y), \label{eq:G3}\\
&\{\partial S(\ove y^n)\} \ \text{is bounded in} \ L^2(0,T;Y^*). \label{eq:G4}
  \end{align}
Then, up to a subsequence, $\haz y^n \weakto y$ in $H^1(0,T;H)$, where
$y$ is a solution of the \emph{GENERIC} system \eqref{gen} in the sense of inequality  \eqref{gen1}.
\end{theorem}

\begin{proof}[Proof of Theorem \ref{thm:cond}]
We infer from the Aubin-Lions lemma \cite[Cor. 7]{Simon87}, upon extracting not relabeled subsequences, that
\begin{align}
  &\haz y^n \weakto y \ \ \text{in} \ \ H^1(0,T;Y), \nonumber \\
& \ove y^n, \, \haz y^n \to y \ \ \text{in} \ \ C([0,T];Y), \nonumber \\
& L(\ove y^n)\,\D E(\ove y^n) \weakto \ell \ \ \text{in} \ \ L^2(0,T;Y), \label{conv1}\\
&\partial S(\ove y^n) \weakto s \ \ \text{in} \ \ L^2(0,T;Y^*). \label{conv2}
\end{align}
As $\partial S$ is strongly-weakly closed, we have that $s=\partial
S(y)$ almost everywhere. On the other hand, the strong-weak continuity
of $\D E$ and the continuity \eqref{eq:contL} of $L$
imply that $L(\ove y^n)\,\D E(\ove y^n) \to L(y)\, \D E(y)$ pointwise
in time,
so that $\ell =  L(y)\, \D E(y)$  almost everywhere.

Fix any $t\in (0,T]$ and let $t^n_m$ be such that $t \in (t^n_{m-1},t^n_m]$.
The discrete sequence of solutions $\{y_i^n\}_{i=0}^{N^n}$ fulfills
\begin{align}
& -S(\ove y^n(t)) + \int_0^{t^n_m} \Psi^*\big(\ove y^n, (\haz y^n)'- L(\ove
 y^n)\, \D E(\ove y^n)\big) + \int_0^{t^n_m} \Psi\big(\ove y^n, \partial
 S(\ove y^n)\big) + S(y^0) \nonumber\\
&\quad = \sum_{i=1}^m G(\tau_i^n , y_{i-1}^n; y_i^n) .\label{to_pass}
\end{align}
As $\Psi(y,\cdot)\geq \Psi(y,0)=0$, we conclude that $\Psi^*\geq 0$ as
well. By restricting integrals to the interval $[0,t] \subset
[0,t^n_m]$ equation \eqref{to_pass} implies that 
\begin{align}
& -S(\ove y^n(t)) + \int_0^{t} \Psi^*(\ove y^n, (\haz y^n)'- L(\ove
 y^n)\, \D E(\ove y^n)) + \int_0^{t} \Psi(\ove y^n, \partial
 S(\ove y^n)) + S(y^0) \nonumber\\
&\quad \leq \sum_{i=1}^{N^n} \big(G(\tau_i^n , y_{i-1}^n; y_i^n)\big)^+. \label{to_pass2}
\end{align}
We now pass to the limit inferior as $n\to \infty$ in relation
\eqref{to_pass2}. By using convergence \eqref{eq:G}, the lower
semicontinuity of $-S$, convergences \eqref{conv1}-\eqref{conv2}, and
the lower semicontinuity \eqref{eq:contK}, we readily obtain that the
limit $y$ fulfills
inequality~\eqref{gen1}.
\end{proof}

The conditional convergence result of Theorem \ref{thm:cond} relies
on the possibility of solving the inequality $G(\tau_i^n , y_{i-1}^n; y_i^n)\leq
0$ up to a small, controllable error, and establishing some a priori
bounds on the discrete solution. The validity of these conditions
has to be checked on the specific problem at hand.  In
the coming Section~\ref{sec:oscillator} we give an example of a
situation where \eqref{eq:G}-\eqref{eq:G4} actually hold.

In case $-S$ is convex, an example of $\{y_i^n\}_{i=0}^{N^n}$
fulfilling \eqref{eq:G} are the solutions of the implicit Euler scheme
\begin{equation}\frac{y_i^n-y_i^n}{\tau_i^n} = L(y_i^n) \, \D E(y_i^n) + K(y_i^n) \, \partial S(y_i^n)
\label{eq:euler}
\end{equation}
whenever available. Indeed, such  $\{y_i^n\}_{i=0}^{N^n}$ fulfills
\begin{align}
G(\tau_i^n , y_{i-1}^n; y_i^n) &  \leq \tau_i^n \Psi^*\left(y_i^n,
  \frac{y_i^n-y_i^n}{\tau_i^n} -  L(y_i^n) \, \D E(y_i^n) \right) +\tau_i^n
\Psi(y_i^n,\partial S(y_i^n))\nonumber\\
&\quad -\lan \partial S(y_i^n), y_i^n - y_{i-1}^n\ran =0,\label{eq:euler2}
\end{align}
where the last equality follows from \eqref{eq:fenchel2} and \eqref{eq:euler}.

Note that the incremental minimization problem \eqref{eq:min} may have
no solutions, even if the Euler scheme is solvable. In the purely dissipative case $L=0$, the
existence of a solution to \eqref{eq:min} is ensured as soon as $\tau^*>0$ exists, so that, for all $\tau\in
(0,\tau^*)$ and $y_{i-1}\in Y$ with $S(y_{i-1})>-\infty$, the function $y \mapsto -S(y) +
\tau\Psi(y,(y-y_{i-1})/\tau)$ is strongly coercive. The latter follows whenever
$-S$ has strongly compact sublevels and  would imply in
particular that the Euler scheme is also solvable. We anticipate however that the example discussed in
Section~\ref{sec:oscillator} is not purely dissipative and
$-S$, albeit convex, does not have strongly compact sublevels.

Let us mention that, in specific applications, the bounds \eqref{eq:G2}-\eqref{eq:G4}
may follow from \eqref{eq:G}.  For instance, this would be the case if
the coercivity 
\begin{equation}\Psi^*(y,\eta)+\Psi(y,\xi)\geq c\| \eta \|_Y^2 + c\|
\xi\|^{2}_{Y^*} -\frac1c\label{eq:coer}
\end{equation}
holds for some $c>0$ and all $y,\, \eta\in Y$ and $\xi \in Y^*$, namely in case $K(y)$ is
positive definite and bounded, uniformly with respect to $y$. This
however does not apply to 
the example of the damped harmonic oscillator from 
Section~\ref{sec:oscillator}, for $K$ is
singular there.

The quadratic nature of the entropy-production potential
could be generalized to the case of $p$-growth with $p>1$ without any
specific intricacy. In particular, one could consider the polynomial
case $\partial \Psi(y,\xi) = K(y) \| \xi\|^{p-2}_{Y^*}\xi$ and
 coercivity \eqref{eq:coer} would then read
$$\Psi^*(y,\eta)+\Psi(y,\xi)\geq c\| \eta \|_Y^{p'}+ c\|
\xi\|^{p}_{Y^*} -\frac1c$$
for $1/p+1/p'=1$.
This setting would correspond to the case of doubly-nonlinear GENERIC
dynamics, namely  
$$ y' = L(y)\, \D E(y) + K(y) \| \partial
S(y)\|^{p-2}_{Y^*}\partial S(y) .$$
Again, under the noninteraction condition \eqref{nic}, suitably regular
solutions conserve
energy, since $\lan \D E(y), \partial \Psi(y,\partial S(y))\ran
= \| \partial S(y)\|_{Y^*}^{p-2}\lan \partial S(y), K^*(y)\, \D
E(y)\ran=0$, and have entropy rate $(\d/\d t)(S(y)) = \| \partial S(y)\|_{Y^*}^{p-2}\lan \partial S(y), K(y)\, \D
S(y)\ran \geq 0$.

Let us further remark that the conditional convergence result of Theorem \ref{thm:cond}
 can serve as an a-posteriori tool to check the convergence
of time-discrete approximations $\{y^n_i\}$, regardless of the method used to
generate them. In particular, relation \eqref{eq:G} can be seen as a
sort of a-posteriori error estimator.

We conclude this section by pointing out that the statement of
Theorem \ref{thm:cond} would hold also for other minimizing-movements
schemes, where nonlinearities are possibly evaluated explicitly. In
particular, instead of $G$ one could consider the functional $\tilde G:
(0,\infty)\times Y \times Y \to (-\infty,\infty]$ defined by 
$$
\tilde G(\tau,\eta;y) = -S(y) +\tau \Psi^*\left(y^1, \frac{y-\eta}{\tau} -
    L(y^2) \, \D E(y^3)\right) \nonumber
	+ \tau \Psi(y^1,\partial S(y)) + S(\eta),
$$
where $y^j$, $j=1,2,3$, are independently chosen to be either $y_i$ or $y_{i-1}$
(or else, for instance $(y_i {+} y_{i-1})/2$). Under conditions
\eqref{eq:G}-\eqref{eq:G4}, convergence would again follow as in
Theorem \ref{thm:cond}.

%%%%%%%%%%%%%%%%%%%%%%%%%%%%%%%%%%%%%%%%%%%%%%%%%%%%%%%%%%%%%%%%%%%%%%%%%%%
\section{The minimizing-movements scheme for the damped harmonic oscillator}\label{sec:oscillator}

In this section, we analyze the simplest GENERIC system fulfilling conditions
\eqref{eq:G}-\eqref{eq:G4}. We consider the case of the damped harmonic
oscillator, namely
\begin{align}
&mq''+\nu q' +\kappa q +\lambda \theta =0,\label{eq:d1}\\
&c\theta' = \nu (q')^2 +\lambda q'\theta. \label{eq:d2}
\end{align}
Here, $q\in \Rz$ represents the position of the harmonic oscillator and $\theta>0$
is its absolute temperature. Note that the case of $\theta$ being the
given, constant
temperature of a surrounding heat bath is considered in
\cite{Oettinger2} instead.
The positive constants $m$, $\nu$, $\kappa$, $\lambda$ and $c$ are
the mass of the oscillator, the viscosity
of the medium, the elastic modulus, the thermal-exchange coefficient, and
the heat capacity, respectively.

By letting $p\in \Rz$ be the momentum of the harmonic oscillator, we
rewrite \eqref{eq:d1}-\eqref{eq:d2} as the first-order system
\begin{align}
  &q' = \frac{p}{m}  , \label{eq:o1}\\
&p'=-\frac{\nu p}{m} - \kappa q -\lambda \theta , \label{eq:o2}\\
&\theta' = \frac{\nu p^2}{m^2c} +\frac{\lambda p\theta}{mc} .\label{eq:o3}
\end{align}

System \eqref{eq:o1}-\eqref{eq:o3} can be written in the GENERIC
form \eqref{gen} by letting   $y=(q,p,\theta)\in Y = \Rz^2\times(0,\infty)$ represent the state of the harmonic oscillator and by defining the total
energy and the total entropy as
 $$E(q,p,\theta) = \frac{p^2}{2m} + c\theta+\frac{\kappa q^2}{2} \ \
 \text{and} \ \ 
 S(q,p,\theta) = -\lambda q + c +c\ln\theta.$$
These definitions comply with the classical Helmholtz relations under
the choice for the free energy
$$\Phi(q,p,\theta) = \frac{\kappa q^2}{2}+\lambda q\theta - c \theta \ln
\theta.$$ 
 In particular, $S =-\partial_\theta \Phi$ and $E = p^2/(2m) + \Phi +
 \theta S$.
Note that both $E$ and $S$ are smooth, $-S$ is convex, and
$S(y)>-\infty$ iff $\theta>0$.  

The operators $L, \, K : \Rz^3 \to \Rz^{3\times 3}$ 
are given by
\begin{equation*}
% \label{eq:K}
L(x)=
  \begin{pmatrix} 
   0 & 1 & 0 \\
   -1 & 0 & -\lambda\theta/c \\
   0 & \lambda\theta/c & 0
  \end{pmatrix}
 \ \ \text{and} \ \  \ K(x)=\nu\theta
  \begin{pmatrix} 
   0 & 0 & 0 \\
   0 & 1 & -p/(mc) \\
   0 & -p/(mc) & p^2/(mc)^2
  \end{pmatrix}.
\end{equation*}
One readily checks that $L$ is antisymmetric and a tedious
computation shows that the Jacobi identity holds. On the other hand, $K$ is
symmetric and positive semidefinite, while not being invertible. Note
that $K(y)$ has rank two for all $p\not =0 $, so that the construction 
in \cite{Oettinger2} does not apply here.

By computing the gradients
\[
 \D E(y)=(\kappa q, p/m,c), \qquad \D S(y)=(-\lambda,0,c/\theta),
\]
one can easily check that the noninteraction condition \eqref{nic}
holds.

Letting $\xi=(\xi_q,\xi_p,\xi_\theta)\in \Rz^3$ and
$\eta=(\eta_q,\eta_p,\eta_\theta) \in \Rz^3$, the entropy-production
potential and its dual read 
\begin{align*}
&\Psi(y,\xi) = \frac{1}{2}\xi \cdot
              K(y)\xi=\frac{\nu\theta}{2}\Bigl(\xi_p-\frac{p}{mc}\xi_\theta\Bigr)^2,\\
&\Psi^*(y,\eta) =
\left\{
\begin{array}{ll}
	\displaystyle
	\frac{1}{2}\frac{\eta_p^2}{\nu\theta} &\quad\text{if} \
                                                       \eta_q =0
                                                       \ \text{and} \
                                                       p \eta_p + mc
                                                       \eta_\theta =0, \\[2mm]
	\displaystyle
	\infty &\quad\text{otherwise.} 
	\end{array}
\right.
\end{align*}
% \begin{align*}
% \Psi(y,\xi) &= \frac{1}{2}\xi \cdot
%               K(y)\xi=\frac{\nu\theta^{K}}{2}\Bigl(\xi_p-\frac{p^{K}}{mc}\xi_\theta\Bigr)^2,\\
% \Psi^*(y,\eta) &=
% \left\{
% \begin{array}{ll}
% 	\displaystyle
% 	\frac{1}{2}\frac{\eta_p^2}{\nu\theta^{K}} &\quad\text{if} \
%                                                        \eta_q =0
%                                                        \ \text{and} \
%                                                        p^K\eta_p + mc
%                                                        \eta_\theta =0, \\
% 	\displaystyle
% 	\infty &\quad\text{otherwise.} 
% 	\end{array}
% \right.
% \end{align*}
% Above and in the following, we indicate with a superscript $K,\,L$, or $E$ the fact
% that that specific occurrence of the variable is related to the
% operators $K,\,L$, or to the functional $E$. Having in mind the
% possibility \eqref{eq:poss} of choosing such dependencies either
% implicitly or explicitly, we keep track of this aspect in the
% following, in order to facilitate the interested reader in the formulation of
% semi-implicit schemes. Let us however anticipate that our anlaysis is
% specifically targeting the fully implicit case only.
 
With the definitions above, the incremental functional
$G:(0,\infty)\times \Rz^3 \times \Rz^3 \to (-\infty,\infty]$ for
$y=(q,p,\theta)$ and $\eta=(\eta_q,\eta_p,\eta_\theta)$ takes the
form 
\begin{align*} 
&G(\tau,\eta;y)=\lambda q -c\ln \theta
+\frac{\tau}{2}\frac{1}{\nu\theta}\Bigl( \frac{p -
  \eta_p}{\tau}+\kappa q+\lambda\theta \Bigr)^2
+\frac{\tau}{2}\nu\theta\Bigr(\frac{p}{m\theta}\Bigl)^2-\lambda
  \eta_q
+c\ln\eta_\theta\\[2mm]
& \quad \text{if}  \ \ \theta>0,\ \ \disp
  \frac{q-\eta_q}{\tau}=\frac{p}{m}, \ \ \text{and} \ \ \disp c \frac{\theta-\eta_\theta}{\tau} = -\frac{p}{m}\frac{p-\eta_p}{\tau}- \frac{\kappa
  pq}{m},\\[2mm]
&G(\tau,\eta;y)=\infty \quad \text{otherwise}.
\end{align*}
%From here on, we choose the fully implicit option $q^E=q_i$,
%$p^E=p^K=p_i$, and
%$\theta^L=\theta^K=\theta_i$. 

Assume to be given the time partition
$0=t_0<\dots<t_N=T$ with $\tau_i=t_i - t_{i-1}$ for $i=1,\dots,N$. By letting
$\eta=(q_{i-1},p_{i-1},\theta_{i-1})$ and defining $(q_0,p_0,\theta_0)
= (q^0,p^0,\theta^0)$ for some given initial state $ (q^0,p^0,\theta^0)\in \Rz^2 \times (0,\infty)$, the incremental minimization
problem \eqref{eq:min} becomes
\begin{align}
 \nonumber
 &\min_{(q_i,p_i,\theta_i)}\Bigl\{\lambda q_i -c\ln \theta_i
   +\frac{\tau_i}{2 \nu\theta_i}\Bigl( \frac{p_i - p_{i-1}}{\tau_i}+\kappa q_i+\lambda\theta_i \Bigr)^2
 +\frac{\tau_i\nu p_i^2}{2m\theta_i}\\ 
 &\qquad \qquad -\lambda q_{i-1}+c\ln\theta_{i-1}\Bigr\}\quad \label{eq:min2}\\[2mm]
&\text{under the constraints}\nonumber\\[2mm]
&\qquad \theta_i >0, \label{eq:c0}\\[1mm]
&
\qquad \frac{q_i-q_{i-1}}{\tau_i}=\frac{p_i}{m},\label{eq:c1}\\
 &\qquad c\frac{\theta_i-\theta_{i-1}}{\tau_i}=-\frac{\kappa q_i p_i}{m}
   -\frac{p_i}{m}\frac{p_i-p_{i-1}}{\tau_i}, \label{eq:c2}\\[1mm]
&\text{for} \ i=1,\dots,N.\nonumber
\end{align}

We devote the remainder of this section to prove that the convergence
result of Theorem \ref{thm:cond} applies to the scheme \eqref{eq:min2}-\eqref{eq:c2},
namely that conditions \eqref{eq:G}-\eqref{eq:G4} are fulfilled. In
particular, we check that the
minimization problem admits a solution $\{y_i\}_{i=0}^N =
\{(q_i,p_i,\theta_i)\}_{i=0}^N $ (Subsection \ref{sec:mini}), that
such solution is unique if time steps are small enough (Subsection \ref{sec:uni}),
that each such solutions  fulfills condition
\eqref{eq:G} in the stronger sense $G(\tau_i,y_{i-1};y_i) \leq 0$
(Subsection \ref{sec:eu}), and that the bounds \eqref{eq:G2}-\eqref{eq:G4}
can be established, independently of the partition (Subsection
\ref{sec:bounds}). Eventually, numerical simulations are presented in Subsection \ref{sec:num}.

%%%%%%%%%%%%%%%%%%%%%%%%%%%%%%%%%%%%%%%%%%%%
\subsection{Existence of minimizers}\label{sec:mini}
Assume to be given $\eta=(q_{i-1},p_{i-1},\theta_{i-1})\in \Rz^2\times(0,\infty).$
As $y \mapsto G(\tau_i,\eta;y)$ is smooth on its domain, in order to
prove that the scheme \eqref{eq:min2}-\eqref{eq:c2} admits a solution,
we just need to check coercivity, namely that sublevels of $y \mapsto
G(\tau_i,\eta;y)$ are bounded. 

By using the constraints \eqref{eq:c1}-\eqref{eq:c2}, we can express
$q_i$ and $\theta_i$ in terms of $p_i$ and the values
$(q_{i-1},p_{i-1}, \theta_{i-1})$ 
 in the form
 \begin{equation}
  q_i = \tau_i \frac{p_i}{m} + q_{i-1} \ \ \text{and} \ \ \theta_i=f(p_i),\label{eq:redefine}
 \end{equation}
 where the function $f$ is defined by
 \[
  f(p_i) = \theta_{i-1}-\frac{\tau_i^2\kappa}{m^2c}p_i^2-\frac{\tau_i\kappa}{mc}p_i q_{i-1}-\frac{1}{mc}p_i^2+\frac{1}{mc}p_i p_{i-1}.
 \]
By substituting the two expressions in the minimum problem
\eqref{eq:min2}-\eqref{eq:c2}, we can reduce it to a minimization in the variable $p_i$
only. Indeed,  problem \eqref{eq:min2}-\eqref{eq:c2} is equivalent to 
\begin{equation}
  \label{eq:min3}
  \min_{p_i} F(p_i) \ \  \text{under the constraint} \ \ f(p_i)>0,
\end{equation}
where we have defined 
\begin{align}
 \nonumber
F(p_i) &:= \frac{\tau_i\lambda p_i}{m}%+\lambda q_{i-1}
   -c\ln f(p_i) +\frac{\tau_i}{2\nu f(p_i)}\Bigl(
   \frac{p_i-p_{i-1}}{\tau_i}+\frac{\tau_i\kappa}{m}  p_i+\kappa q_{i-1}
   +\lambda f(p_i)\Bigr)^2\nonumber\\
   &\phantom{:xx} + \frac{\tau_i\nu p_i^2}{2m f(p_i)} %-\lambda q_{i-1}
   +c\ln\theta_{i-1}. \nonumber
\end{align}
Again, as $F$ is smooth on its domain,  in order to solve the
minimization problem \eqref{eq:min3} we just need to prove that $F$ has bounded sublevels. This
follows from the fact that $F(p_i)\to \infty$ if $f(p_i)\to 0^+$ and
that $f(p_i)>0$ iff $p_i$ belongs to a bounded interval, depending on
$(q_{i-1},p_{i-1},\theta_{i-1})$. In fact, as $f$ is quadratic and
$\theta_{i-1}>0$, it follows that $f(p_i)>0$ if and only if
$|p_i - \haz p|<\haz r$ with 
\begin{align*}
&\haz p = \frac{p_{i-1}-\tau_i \kappa q_{i-1}}{2+2\tau_i^2 \kappa/m
  }  \ \ \text{and} \ \ 
\haz r=\frac{\left({(p_{i-1}{-}\tau_i\kappa
  q_{i-1})^2}   + 
  4(1{+}\tau_i^2\kappa/m)mc\theta_{i-1}\right)^{1/2} }{2+2\tau_i^2\kappa/m}  
 .
\end{align*}
Hence, for all given $(q_{i-1},p_{i-1},\theta_{i-1})\in\Rz^2\times
(0,\infty)$, we can find a solution $p_i$ of \eqref{eq:min3}. 
 We conclude from \eqref{eq:redefine} that
$(q_i,p_i,\theta_i)$ solves \eqref{eq:min2}-\eqref{eq:c2}. In particular, we have
that $\theta_i>0$.

%%%%%%%%%%%%%%%%%%%%%%%%%%%%%%%%%%%%%%%%
\subsection{Uniqueness of minimizers}\label{sec:uni}
 Let us remark that the minimizers
$(q_i,p_i,\theta_i)$ of \eqref{eq:min2}-\eqref{eq:c2} need not be unique in general. Still, the
function $F$ is strictly convex for $\tau_i$
small enough, depending on the values $(q_{i-1},p_{i-1},\theta_{i-1})$
and on material parameters. 
\newpage
Let us start by rewriting $F$ as
% \UUU
% \begin{align}
%   F(p_i)&= -c \ln f(p_i) + \frac{(1+\tau_i
%           \kappa)^2 }{2\tau_i \nu  }\frac{ (p_i+h_i)^2}{  f(p_i)} + \frac{\tau_i\lambda^2f(p_i)}{2\nu}+ r_i(p_i)\label{todiff}
% \end{align}
% where $h_i =(-p_{i-1}+\tau^i \kappa q_{i-1})/(1+\tau_i^2\kappa)$ and $r_i$
% is the affine function of $p_i$
% $$r_i(p_i)=  \frac{\tau_i \lambda p_i}{m}
% + \frac{\lambda (1+ \tau_i^2 \kappa)}{\nu }(p_i + h_i)
%  + c \ln \theta_{i-1}.$$
% 
\begin{align}
  F(p_i)&= -c \ln f(p_i) + \frac{(1+\tau_i^2 
          \kappa/m)^2 }{2\tau_i \nu  }\frac{ (p_i+h_i)^2}{  f(p_i)}
          \nonumber\\
&\quad + \frac{\tau_i \nu}{2m}\frac{p_i^2}{f(p_i)}
          + \frac{\tau_i\lambda^2f(p_i)}{2\nu}+ r_i(p_i), \label{todiff}
\end{align}
where $h_i =(-p_{i-1}+\tau_i \kappa q_{i-1})/(1+ \tau_i^2 \kappa/m)$ and $r_i$
is the affine function
$$r_i(p_i)=  \frac{\tau_i \lambda p_i}{m}
+ \frac{\lambda (1+ \tau_i^2 \kappa/m)}{\nu }(p_i + h_i)
 + c \ln \theta_{i-1}.$$
Using  the following facts 
\begin{align*}
 \left(\frac{(p_i+h_i)^2}{f(p_i)}\right)''  &=  \frac{\big(\big(2
                                              f^3(p_i)\big)^{1/2} -\big(2
    f(p_i)\big)^{1/2}(p_i+h_i) f'(p_i)\big)^2 }{  f^4(p_i)}\\
& \,\quad  - \frac{ (p_i+h_i)^2
    f^2(p_i)f''(p_i)}{  f^4(p_i)}> 0,\\
 \left(\frac{p_i^2}{f(p_i)}\right)''  &=  \frac{\big(\big(2
                                              f^3(p_i)\big)^{1/2} -\big(2
    f(p_i)\big)^{1/2} p_i  f'(p_i)\big)^2 }{  f^4(p_i)}\\
& \,\quad  - \frac{ p_i ^2
    f^2(p_i)f''(p_i)}{  f^4(p_i)}> 0, 
%\\[2mm]
%\left(\frac{1}{f(p_i)}\right)'' &=  \frac{-f''(p_i)f^2(p_i) + 2 f(p_i)(f'(p_i))^2}{f^4(p_i)}>0,
\end{align*}
and differentiating $F$ in \eqref{todiff} twice, we obtain
\begin{align*}
  F''(p_i) &> -c (\ln f(p_i))''  +\frac{\tau_i
    \lambda^2}{2\nu}f''(p_i)  = \frac{c(f'(p_i))^2}{f^2(p_i)} +
    f''(p_i)\left(- \frac{c}{f(p_i)} +  \frac{\tau_i
    \lambda^2}{2\nu}\right)\\
&\geq f''(p_i)\left(- \frac{c}{f(p_i)} +  \frac{\tau_i
    \lambda^2}{2\nu}\right)=:g(p_i).
\end{align*}
 As $f''(p_i) =
-2(1+\tau_i^2 \kappa/m)/(mc)<0$, the function $g$ is nonnegative as long
as $f(p_i)\leq 2\nu c/(\tau_i \lambda^2)$. It hence suffices to choose
$\tau_i$ so small that 
\begin{equation}
  \label{eq:small}
 \max f \leq \frac{2\nu c}{\tau_i \lambda^2}.
\end{equation}
 This condition depends on the size of the time-step  
and on material parameters and
implies that $F$ is strictly convex, thus admitting a unique minimizer
$p_i$. Consequently, problem \eqref{eq:min2}-\eqref{eq:c2} has a unique
minimizer $(q_i,p_i,\theta_i)$.

We can make the time-step dependent bound \eqref{eq:small} on $\tau_i$
more explicit by directly computing the
maximum of $f$:
% $$\max f = \frac{(p_{i-1}-\tau_i\kappa q_{i-1})^2 + 4(1+\tau_i^2\kappa)mc\theta_{i-1}}{4(1+\tau_i^2\kappa)mc}\leq
% \frac{p_{i-1}^2}{2mc} + \frac{\tau_i^2 \kappa^2 q_{i-1}^2}{2mc} +
% \theta_{i-1}. $$
% 

\begin{align*}
\max f &= \frac{\left(\displaystyle \frac{-\tau_i\kappa q_{i-1}}{mc} +
    \frac{p_{i-1}}{mc}\right)^2+4\left( \displaystyle \frac{\tau_i^2\kappa}{m^2c} +
    \frac{1}{mc}\right) \theta_{i-1}}{4\left(\displaystyle 
    \frac{\tau_i^2\kappa}{m^2c} + \frac{1}{mc}\right)}\\
&=
\frac{(-\tau_i \kappa q_{i-1} + p_{i-1})^2}{4(\tau_i^2 \kappa c + mc)} + \theta_{i-1}
% \frac{(p_{i-1}-\tau_i\kappa q_{i-1})^2
% +4(1+\tau_i^2\kappa/m)mc\theta_{i-1}}{ 4  (1+\tau_i^2\kappa/m)}
\leq
\frac{p_{i-1}^2}{2mc} + \frac{\tau_i^2 \kappa^2 q_{i-1}^2}{2mc} +
\theta_{i-1}. 
\end{align*}

In Subsection \ref{sec:bounds} we will check that above right-hand
side can be
bounded by a positive constant $c_0$ in terms of the initial data and of material parameters only, uniformly with
respect to $i$, see \eqref{eq:bound1}. 
Condition
\eqref{eq:small} can hence be strengthened by asking the partition to
be such that  
\begin{equation} \label{eq:small2}
c_0 \leq \min_i\frac{2\nu c}{\tau_i
  \lambda^2}.
\end{equation} 
This stronger condition ensures the
unique solvability of the minimization problem \eqref{eq:min2}-\eqref{eq:c2} for all $i=1,\dots,N$.

%%%%%%%%%%%%%%%%%%%%%%%%%%%%%%%%%%%%%%%%%%%%
\subsection{Minimizers fulfill condition \eqref{eq:G}}\label{sec:eu}
Following the discussion of Section \ref{sec:gen}, the convexity of
$-S$ ensures that 
condition \eqref{eq:G} holds as soon as one can prove that the
implicit Euler scheme \eqref{eq:euler} is solvable. In the current
setting, the implicit Euler scheme reads
\begin{align}
	&\displaystyle
	 \frac{q_i-q_{i-1}}{\tau_i}=\frac{p_i}{m}, \label{eq:Eimp1}\\[1mm]
&	\displaystyle
	 \frac{p_i-p_{i-1}}{\tau_i}=-\frac{\nu p_i}{m}-\kappa q_i
          -\lambda \theta_i, \label{eq:Eimp2}\\[1mm]
&	\displaystyle
	 \frac{\theta_i-\theta_{i-1}}{\tau_i}=\frac{\nu
          p_i^2}{m^2c}+\frac{\lambda p_i\theta_i}{mc} \label{eq:Eimp3}
\end{align}
for $i=1,\dots,N$.

Let the initial data $(q^0,p^0,\theta^0)\in \Rz^2 \times (0,\infty)$ 
 be given and assume to have solved \eqref{eq:Eimp1}-\eqref{eq:Eimp3} up to
level $i-1$. In particular, let
$(q_{i-1},p_{i-1},\theta_{i-1}) \in \Rz^2 \times (0,\infty)$.
 Relation \eqref{eq:Eimp1} can be written as
 \[
  q_i=\frac{\tau_i p_i}{m}+q_{i-1}
 \]
 and we can substitute it into  \eqref{eq:Eimp2}  obtaining
 \begin{equation}
  p_i = \frac{-\tau_im\lambda}{m+\tau_i\nu+\tau^2_i\kappa} \theta_i   +
  \frac{-\tau_im\kappa
  q_{i-1}+mp_{i-1}}{m+\tau_i\nu+\tau^2_i\kappa}
=:\alpha_i\theta_i+\beta_i \label{eq:pi}
 \end{equation}
where $\alpha_i$ and $\beta_i$ depend just on material parameters, $\tau_i$, and on $q_{i-1}$
and $p_{i-1}$. 

We may hence rewrite equation \eqref{eq:Eimp3} as 
 \begin{align*}
  \theta_i &=\theta_{i-1}+\frac{\tau_i\nu p_i^2}{m^2c} +\frac{\tau_i \lambda p_i\theta_i}{mc}\\
  &=\theta_{i-1}+\frac{\tau_i\nu}{m^2c}\Bigl(\alpha^2_i\theta_i^2 +2\alpha_i\beta_i\theta_i +\beta^2_i\Bigr) +\frac{\tau_i\lambda}{mc}\Bigl(\alpha_i\theta_i^2 +\beta_i\theta_i\Bigr),  
 \end{align*}
 which gives the expression
 \begin{align}
0&=\Bigl(\frac{\tau_i \nu \alpha_i^2}{m^2c} + \frac{\tau_i \lambda
    \alpha_i}{mc} \Bigr)\theta_i^2 + \Bigl(\frac{2\tau_i \nu
    \alpha_i\beta_i}{m^2c}+ \frac{\tau_i \lambda
    \beta_i}{mc}-1\Bigr)\theta_i + \Bigl( \theta_{i-1}+\frac{\tau_i \nu
    \beta_i^2}{m^2c}\Bigl)\nonumber\\
&=:\gamma_i\theta_i^2 + \delta_i\theta_i +\epsilon_i, \label{eq:zero}
 \end{align}
with $\gamma_i$, $\delta_i$, and $\epsilon_i$ depending on $\alpha_i$,
$\beta_i$, $\theta_{i-1}$, and data. Note that %because of $\alpha_i<0$, 
% $$
%   \gamma_i =
%   \frac{\tau_i}{mc}\bigg(\frac{\nu\alpha_i}{m}+\lambda\bigg)\alpha_i\red{=
%   - \frac{\tau_i^2\lambda^2}{mc} \frac{m+\tau_i^2\kappa}{m+\tau_i\nu+\tau_i^2\kappa}}<0.%\frac{\tau_i\lambda}{mc}\,\frac{m+\tau_i^2\kappa}{m+\tau_i\nu+\tau_i^2\kappa}<0.
% $$

$$
  \gamma_i =
  \frac{\tau_i}{mc}\bigg(\frac{\nu\alpha_i}{m}+\lambda\bigg)\alpha_i =
  - \frac{\tau_i^2\lambda^2}{c} \frac{m+\tau_i^2\kappa}{(m+\tau_i\nu+\tau_i^2\kappa)^2}<0.
$$
 We infer from $\theta_{i-1}>0$ that $\epsilon_i>0$. Hence, the
second-order equation \eqref{eq:zero} has a unique solution
$\theta_i>0$. This uniquely defines $p_i$ and $q_i$ via \eqref{eq:pi}
and \eqref{eq:Eimp1}, respectively.  In particular, the implicit Euler
scheme  \eqref{eq:Eimp1}-\eqref{eq:Eimp3} has a unique solution $y^e_i
=(q_i,p_i,\theta_i)$. 
 
Let now $y_i $ be a solution to the incremental minimization
problem \eqref{eq:min2}. Owing to \eqref{eq:euler2}, we conclude that 
$$
G(\tau_i,y_{i-1};y_i) = \min_y G(\tau_i,y_{i-1};y) \leq
G(\tau_i,y_{i-1};y_i^e) \stackrel{\eqref{eq:euler2}}{=}0,
$$
where we have also used  the fact that  $y^e_i $ fulfills the
constraints \eqref{eq:c0}-\eqref{eq:c2}. In particular, condition \eqref{eq:G}
holds in the even stronger form 
\begin{equation} G(\tau_i,y_{i-1};y_i) \leq 0 \quad \text{for} \ \
i=1,\dots,N. \label{eq:stronger}
\end{equation}

%%%%%%%%%%%%%%%%%%%%%%%%%%%%%%%%%%%%%%%%%%%%
\subsection{A priori bounds}\label{sec:bounds}
We now prove that condition \eqref{eq:G2}-\eqref{eq:G4} of Theorem
\ref{thm:cond} hold. This will follow by checking a priori bounds on minimizers $\{y_i\}_{i=0}^N$ of
\eqref{eq:min2}, independently from the time partition. 

Let us start by remarking that
the energy is nonincreasing. By \eqref{eq:c1}, one can equivalently rewrite the constraint \eqref{eq:c2}
as
 \begin{align*}
  c\theta_i-c\theta_{i-1} =-\frac{\kappa q_i^2}{2}+\frac{\kappa
   q_{i-1}^2}{2} -\frac{\kappa}{2}|q_i-q_{i-1}|^2 -\frac{p_i^2}{2m} +\frac{p^2_{i-1}}{2m}-\frac{1}{2m}|p_i-p_{i-1}|^2.
 \end{align*}
A rearrangement of the terms leads to
$$E(y_i) + \frac{\kappa}{2}|q_i-q_{i-1}|^2+\frac{1}{2m}|p_i-p_{i-1}|^2 =
E(y_{i-1}).$$
Hence,   $  E(y_i)$ is nonincreasing and $E(y_i) =E(y_{i-1})$ iff
$q_i=q_{i-1}$ and $p_i=p_{i-1}$. By taking the sum for $i=1,\dots,j$ for
$j\leq N$, we find that
\begin{equation}
  \label{eq:energy}
  E(y_j) + D^q_j +D^p_j=E(y^0) \quad \forall j =1,\dots,N,
\end{equation}
where we have set  
$$D^q_j
=\sum_{i=1}^j\frac{\kappa}{2}|q_i-q_{i-1}|^2  \ \ \text{and} \ \ D^p_j
=\sum_{i=1}^j\frac{1}{2m}|p_i-p_{i-1}|^2.$$ 
The nonnegative terms
$D^q_j$ and $D^p_j$ exactly quantify the energy dissipativity of the
scheme. Owing to \eqref{eq:energy}, we obtain the uniform bound 
\begin{equation}
 |p_i| + |\theta_i| +  |q_i|\leq
C \quad \forall i=1,\dots,N,\label{eq:bound1}
\end{equation}
where, here and in the following, the symbol $C$ denotes a generic
positive constant depending 
on the initial data $y^0$ and on material
parameters, but not on the time partition. Bound \eqref{eq:bound1} and constraint
\eqref{eq:c1} imply that 
\begin{equation}
 \left|\frac{q_i-q_{i-1}}{\tau_i}\right|\leq
C \quad \forall i=1,\dots,N.\label{eq:bound2}
\end{equation}
Moreover, we readily check that 
\begin{equation}
  |L(y_i)\, \D E(y_i)|= |(p_i/m,-kq_i-\lambda \theta_i,\lambda
  p_i\theta_i/m)| \stackrel{\eqref{eq:bound1}}{\leq} C \quad \forall i=1,\dots,N.\label{eq:bound20}
\end{equation}

Let us take  the sum for $i=1,\dots,j$ for $j \leq N$ in \eqref{eq:stronger} getting
$$ - S(y_j) + \psi^*_j+ \psi_j \leq -S(y^0),$$
where 
$$\psi^*_j = \sum_{i=1}^j \frac{\tau_i}{2 \nu\theta_i}\Bigl( \frac{p_i - p_{i-1}}{\tau_i}+\kappa q_i+\lambda\theta_i \Bigr)^2
  \ \ \text{and}   \ \ \psi_j=\sum_{i=1}^j \frac{\tau_i\nu
    p_i^2}{2m\theta_i}.$$
As $\psi^*_j$ and $\psi_j$ are nonnegative, we infer that $  S(y_i)$
is nondecreasing. In particular,
$$ - c\ln\theta_j \leq -S(y^0) - \lambda q_j +c
\stackrel{\eqref{eq:bound1}}{\leq }C \quad \forall j=1,\dots,N. $$
and consequently,
\begin{equation}
  \label{eq:bound3}
  \theta_i \geq \theta_{\rm min}>0 \quad \forall i=1,\dots,N, 
\end{equation}
 for some given $\theta_{\rm min}$ depending just on $y^0$ and the material
parameters. This ensures that 
\begin{equation}
  \label{eq:bound30}
  |\D S(y_i)| = |(-\lambda,0,c/\theta_i)| \leq C \quad \forall i=1,\dots,N.
\end{equation}

It follows from the bound \eqref{eq:bound1} on $q_i$ and
$\theta_i$ that $S(y_i) =
-\lambda q_i + c + c \ln \theta_i \leq
C$ for all $i=1,\dots,N$. We conclude that $\psi_N \leq
-S(y^0)+S(y_N)\leq C$ and
\begin{align}
&\sum_{i=1}^N  \frac{\tau_i}{2\nu
  \theta_i} \left|\frac{p_i - p_{i-1}}{\tau_i} \right|^2 \nonumber\\
&\quad= \psi_N^* - \sum_{i=1}^N  \frac{\tau_i}{2\nu
  \theta_i} \left(\kappa q_i +\lambda\theta_i\right)^2 - 2 \sum_{i=1}^N  \frac{\tau_i}{2\nu
  \theta_i} \left(\frac{p_i - p_{i-1}}{\tau_i} \right) \left(\kappa
  q_i +\lambda\theta_i\right)\nonumber\\
&\quad\leq \psi_N^*  + \sum_{i=1}^N  \frac{\tau_i}{2\nu
  \theta_i} \left(\kappa q_i +\lambda\theta_i\right)^2
  +\sum_{i=1}^N\frac{\tau_i}{4\nu \theta_i} \left|\frac{p_i - p_{i-1}}{\tau_i}
  \right|^2 \nonumber\\
&\quad 
\stackrel{\eqref{eq:bound1}}{\leq} C  +\sum_{i=1}^N  \frac{\tau_i}{4\nu
  \theta_i} \left|\frac{p_i - p_{i-1}}{\tau_i} \right|^2.\label{eq:pp}
\end{align}
Using the fact that $\theta_i$ is
uniformly bounded by \eqref{eq:bound1}, we can hence bound
\begin{align}
  \sum_{i=1}^N \tau_i \left|\frac{p_i - p_{i-1}}{\tau_i} \right|^2
  \leq  4\nu \,(\max_i\theta_i) \sum_{i=1}^N  \frac{\tau_i}{4\nu
  \theta_i} \left|\frac{p_i - p_{i-1}}{\tau_i} \right|^2
  \stackrel{\eqref{eq:bound1}+\eqref{eq:pp}}{\leq}  C.\label{eq:bound4}
\end{align}
Eventually, we infer from constraint \eqref{eq:c2} that 
\begin{equation}
  \label{eq:bound5}
  \sum_{i=1}^N \tau_i \left|\frac{\theta_i - \theta_{i-1}}{\tau_i}
  \right|^2 \leq  2\sum_{i=1}^N \tau_i \left(\frac{\kappa^2 q_i^2
      p_i^2}{m^2c^2} + \frac{p_i^2}{m^2c^2}\left| \frac{p_i-p_{i-1}}{\tau_i}\right|^2 \right)\leq C.
\end{equation}

Bounds \eqref{eq:bound1}-\eqref{eq:bound2} and
\eqref{eq:bound4}-\eqref{eq:bound5} imply condition
\eqref{eq:G2} from Theorem \ref{thm:cond}. On the other hand, bounds
\eqref{eq:bound20} and \eqref{eq:bound30} imply conditions
\eqref{eq:G3} and \eqref{eq:G4}, respectively. Hence, the convergence
statement of Theorem
\ref{thm:cond} holds. In particular, given initial values
$(q^0,p^0,\theta^0)\in \Rz^3$ with $\theta^0>0$, a 
sequence of partitions $0=t_0^n<\dots<t^n_{N^n}=T$ with $\tau^n=\max_i( t^n_i
- t^n_{i-1})\to 0$, and corresponding minimizers $\{(q_i^n,p_i^n,
\theta_i^n)\}_{i=0}^N$ of \eqref{eq:min2}, it follows that 
\begin{equation}
(\haz q^n,\haz p^n,\haz \theta^n) \weakto (q,p,\theta) \ \
\text{in} \ \ H^1(0,T;\Rz^3),\label{eq:conf}
\end{equation}
where $ (q,p,\theta)$ solves the damped harmonic oscillator system
\eqref{eq:o1}-\eqref{eq:o3} with initial value $
(q(0),p(0),\theta(0))= (q^0,p^0,\theta^0)$. Note that solutions of
\eqref{eq:o1}-\eqref{eq:o3} are
unique. Hence, convergence \eqref{eq:conf}
holds for the whole sequence of
discrete solutions, not just for a subsequence.

As $E$ and $S$ are smooth, $\ove y^n$ is bounded, and $\theta_i \geq \theta_{\rm min}>0$, we find that 
$E(\ove y(\cdot)) \to E(y(\cdot))\equiv E(y^0)$ and $ S(\ove y(\cdot))
\to S(y(\cdot))$ uniformly, together with their time derivatives of all orders.
More precisely, upon remarking that 
\begin{align*}
&D_N^q \leq \tau^n\sum_{i=1}^N\frac{\tau_i\kappa}{2} \left|\frac{q_i -
    q_{i-1}}{\tau_i} \right|^2\stackrel{\eqref{eq:bound2}}{\leq}
C\tau^n, \\
&D_N^p  \leq \tau^n\sum_{i=1}^N\frac{\tau_i}{2m} \left|\frac{p_i -
    p_{i-1}}{\tau_i} \right|^2\stackrel{\eqref{eq:bound4}}{\leq}
C\tau^n,
\end{align*}
we may use \eqref{eq:energy} and check that 
$$ -C\tau^n + E(\ove y^n(t)) \leq E^0 \quad \forall t \in [0,T].$$
In particular, we control the energy dissipation as follows:
\begin{equation}
\max_{t\in [0,T]}\| E(\ove y^n(t)) -  E(y^0)\|_{L^\infty(0,T)} \leq C\tau^n.
\label{errore}
\end{equation}

%%%%%%%%%%%%%%%%%%%%%%%%%%%%%%%%%%%%%%%%%%%%
\subsection{Numerical tests}\label{sec:num}

We record here some numerical evidence on the performance of the
minimizing-movements scheme \eqref{eq:min2}-\eqref{eq:c2}. All computations are
performed in Matlab. In the following, we choose 
\begin{equation}
m=\nu=\kappa=\lambda=c=1, \quad y^0 = (q^0,p^0,\theta^0)= (1,1,1),
\quad T=15.\label{data}
\end{equation}
Given a
uniform time-partition with time step $\tau=T/N$, we find a
solution $\{y_i\}_{i=0}^N$ of the minimizing-movements scheme
\eqref{eq:min2}-\eqref{eq:c2} via Newton's method. Note that, by
choosing \eqref{data} and recalling the bound \eqref{eq:energy}, one
has that the constant $c_0$ in \eqref{eq:small2} can be bounded from above
by $E(y^0)=2$, for all $\tau\leq 1$. In particular, condition
\eqref{eq:small2} is fulfilled by all $\tau\leq 1$ and the solution of
the minimizing-movements scheme is unique.

 The unique solution $\{y^e_i\}_{i=0}^N$ of the
implicit Euler scheme \eqref{eq:Eimp1}-\eqref{eq:Eimp3} is obtained
directly from \eqref{eq:pi}-\eqref{eq:zero}. Eventually, the numerical
reference
solution $t\mapsto y(t)$ is calculated by means of the Matlab solver
ode45 choosing $10^{-4}$ for the maximal time step and $10^{-8}$ for the absolute tolerance.

Figures~\ref{q1}-\ref{entropy1} illustrate the
numerical
reference solution and the time-discrete solutions for
$\tau=1/4$. Both the minimizing-movements and the  Euler scheme
dissipate energy, see Figure~\ref{entropy1} left. This is of
course an undesired effect, which is however attenuated as $\tau \to
0$, see \eqref{errore}. Energy dissipation seems to be more pronounced
for the Euler scheme. On the other hand, entropy is nondecreasing for
the minimizing-movements scheme whereas the Euler scheme shows a nonmonotone
entropy, which is nonphysical, see Figure~\ref{entropy1} right. 
%  
% \begin{figure}[ht]
% \centering
% \pgfdeclareimage[width=65mm]{q1}{figures/q1} 
% 
% \pgfdeclareimage[width=65mm]{p1}{figures/p1} 
% \pgfuseimage{q1}\pgfuseimage{p1}
% \caption{Position $q$ (left)  and momentum $p$ (right) with respect to time for the numerical reference solution (dotted),
%   the minimizing-movements scheme (solid), and the Euler scheme (dash-dotted).}
% \label{q1}
% \end{figure}

\begin{figure}[ht]
\centering
\includegraphics[width=60mm]{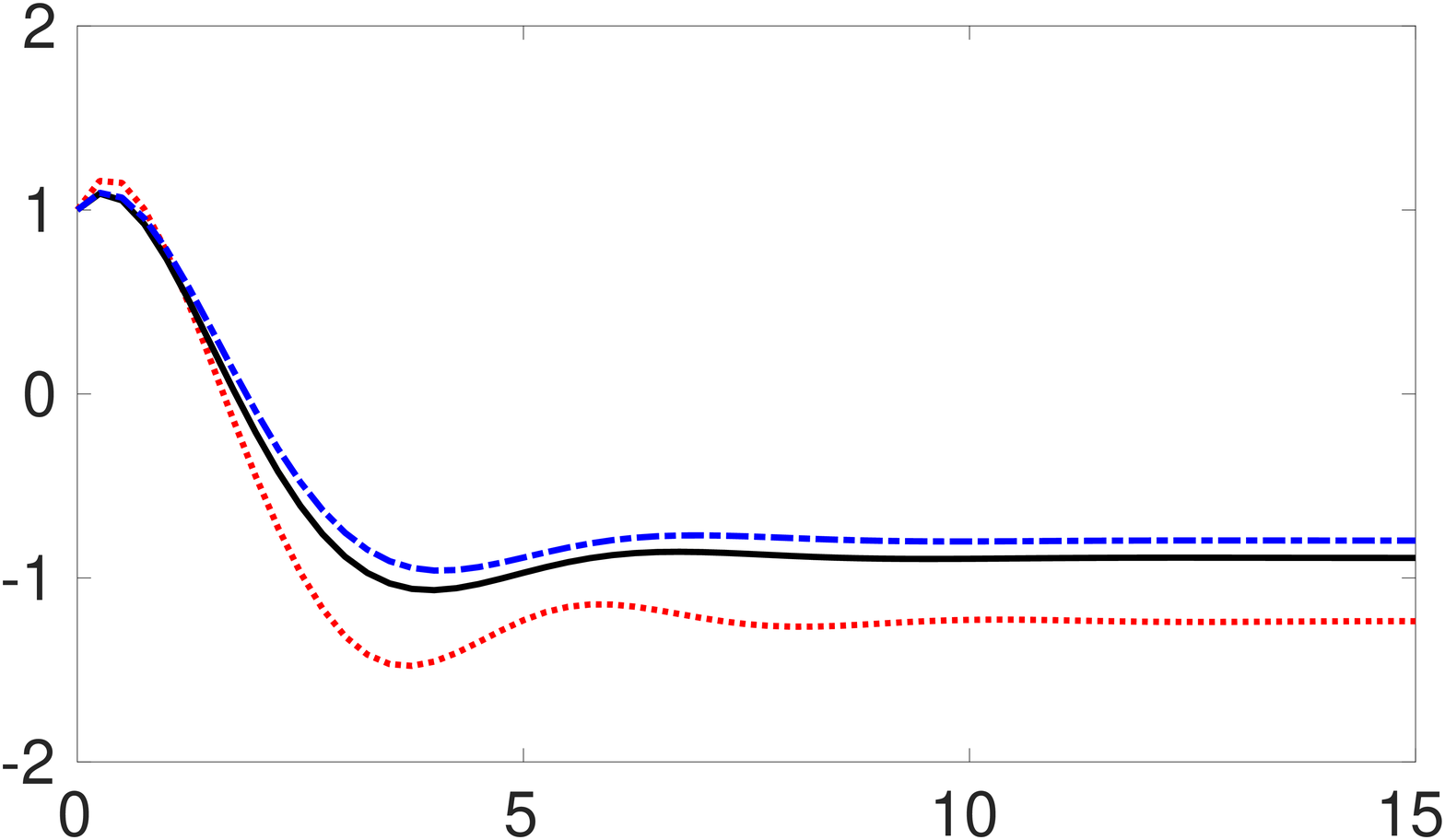} 
\includegraphics[width=60mm]{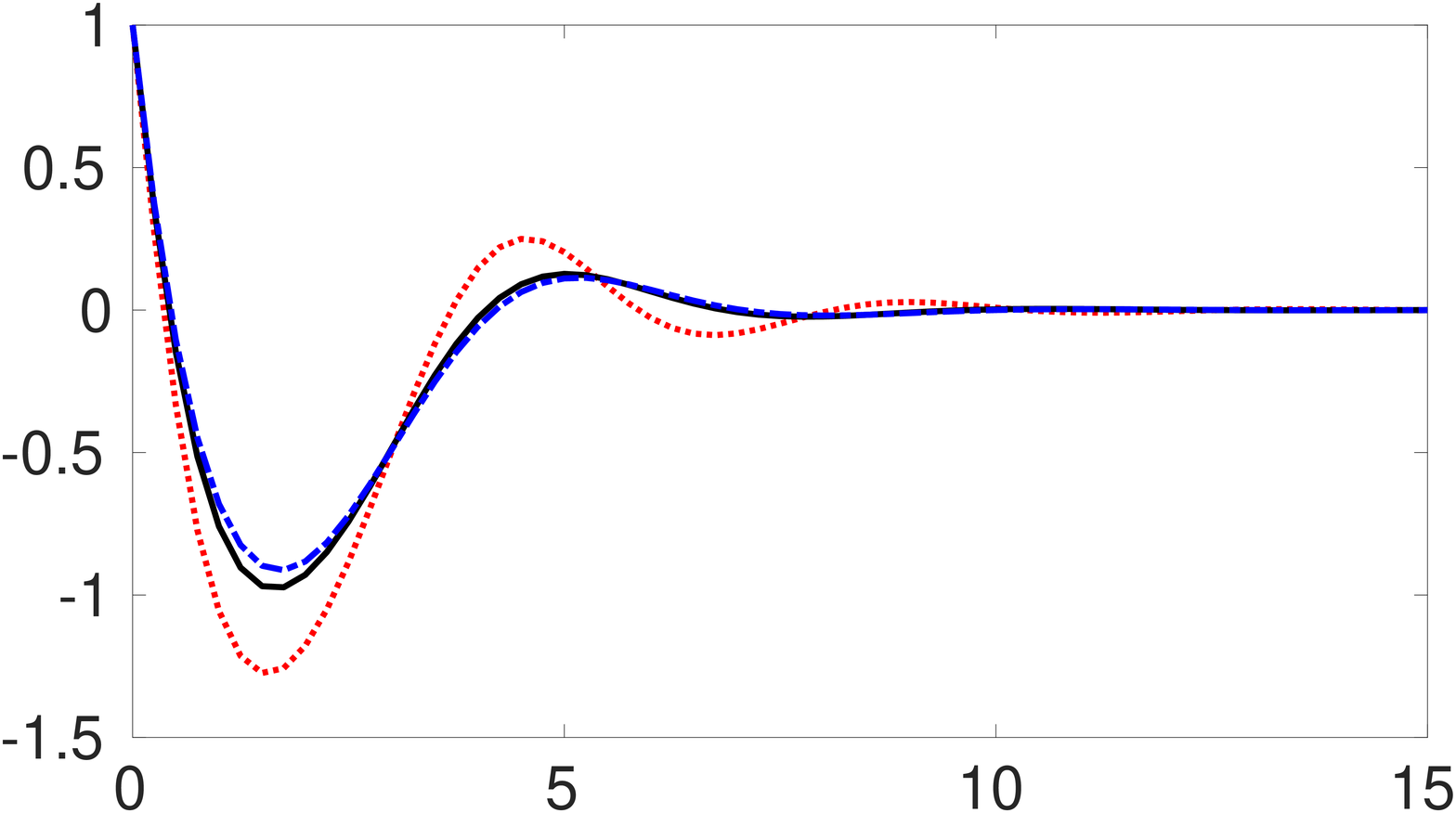} 
\caption{Position $q$ (left)  and momentum $p$ (right) with respect to time for the numerical reference solution (red, dotted),
  the minimizing-movements scheme (solid), and the Euler scheme (blue, dash-dotted).}
\label{q1}
\end{figure}
% 
% \begin{figure}[ht]
% \centering
% \pgfdeclareimage[width=65mm]{theta1}{figures/theta1} 
% \pgfuseimage{theta1}
% \caption{Temperature with respect to time for the numerical reference solution (dotted),
%   the minimizing-movements scheme (solid), and the Euler scheme (dash-dotted).}
% %\label{t1}
% \end{figure}

\begin{figure}[ht]
\centering
\includegraphics[width=60mm]{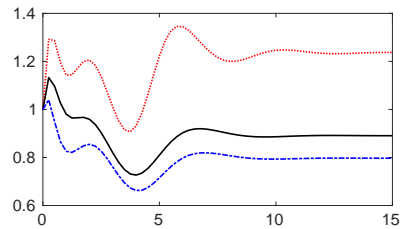} 
\caption{Temperature with respect to time for the numerical reference solution (red, dotted),
  the minimizing-movements scheme (solid), and the Euler scheme (blue, dash-dotted).}
%\label{t1}
\end{figure}

% \begin{figure}[ht]
% \centering
% \pgfdeclareimage[width=65mm]{energy1}{figures/energy1} 
% \pgfdeclareimage[width=65mm]{entropy1}{figures/entropy1} 
% \pgfuseimage{energy1}\pgfuseimage{entropy1}
% \caption{Energy (left) and entropy (right) as function of time for the numerical reference solution (dotted),
%   the minimizing-movements scheme (solid), and the Euler scheme (dash-dotted).}
% \label{entropy1}
% \end{figure}

\begin{figure}[ht]
\centering
\includegraphics[width=60mm]{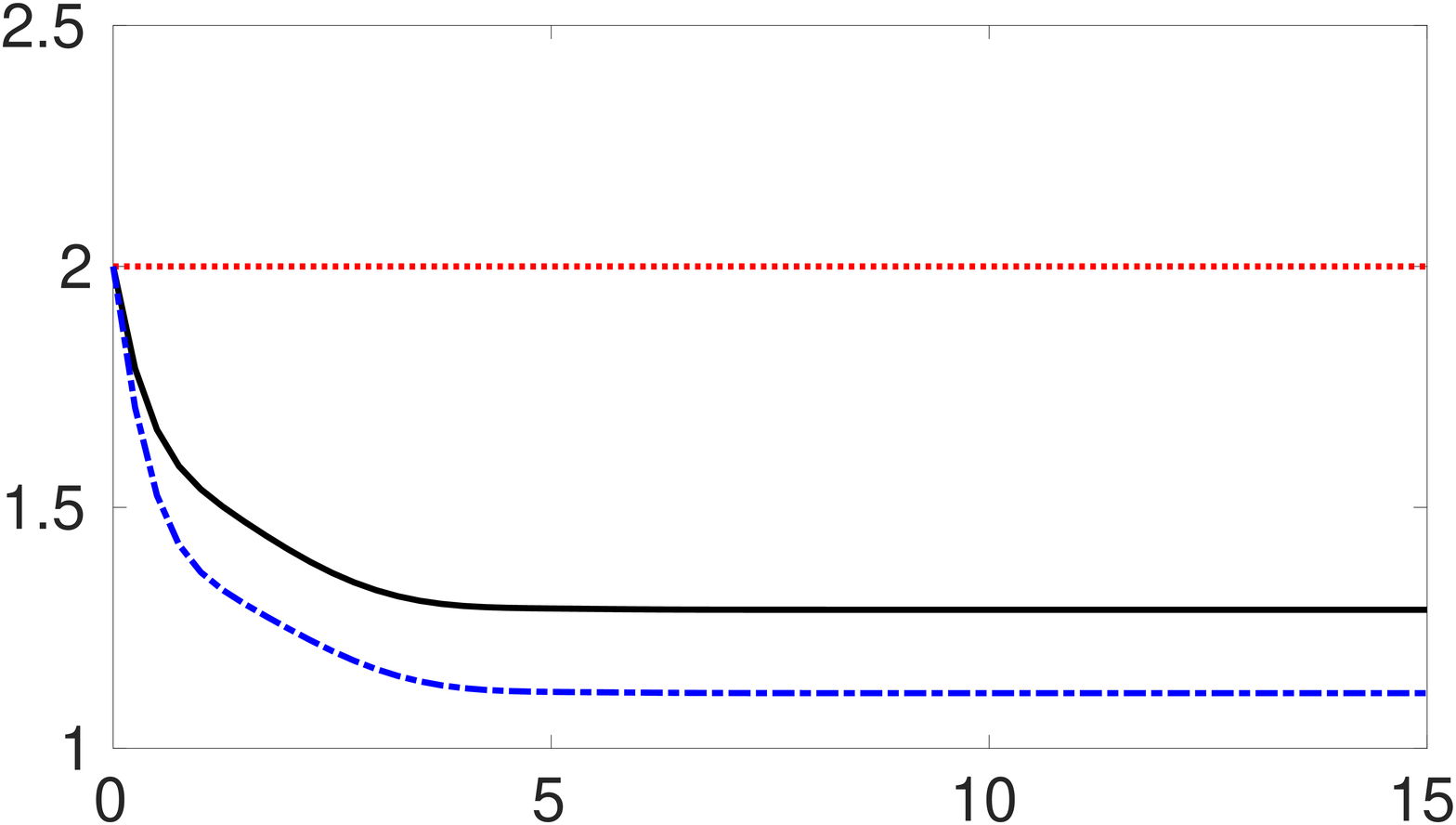} 
\includegraphics[width=60mm]{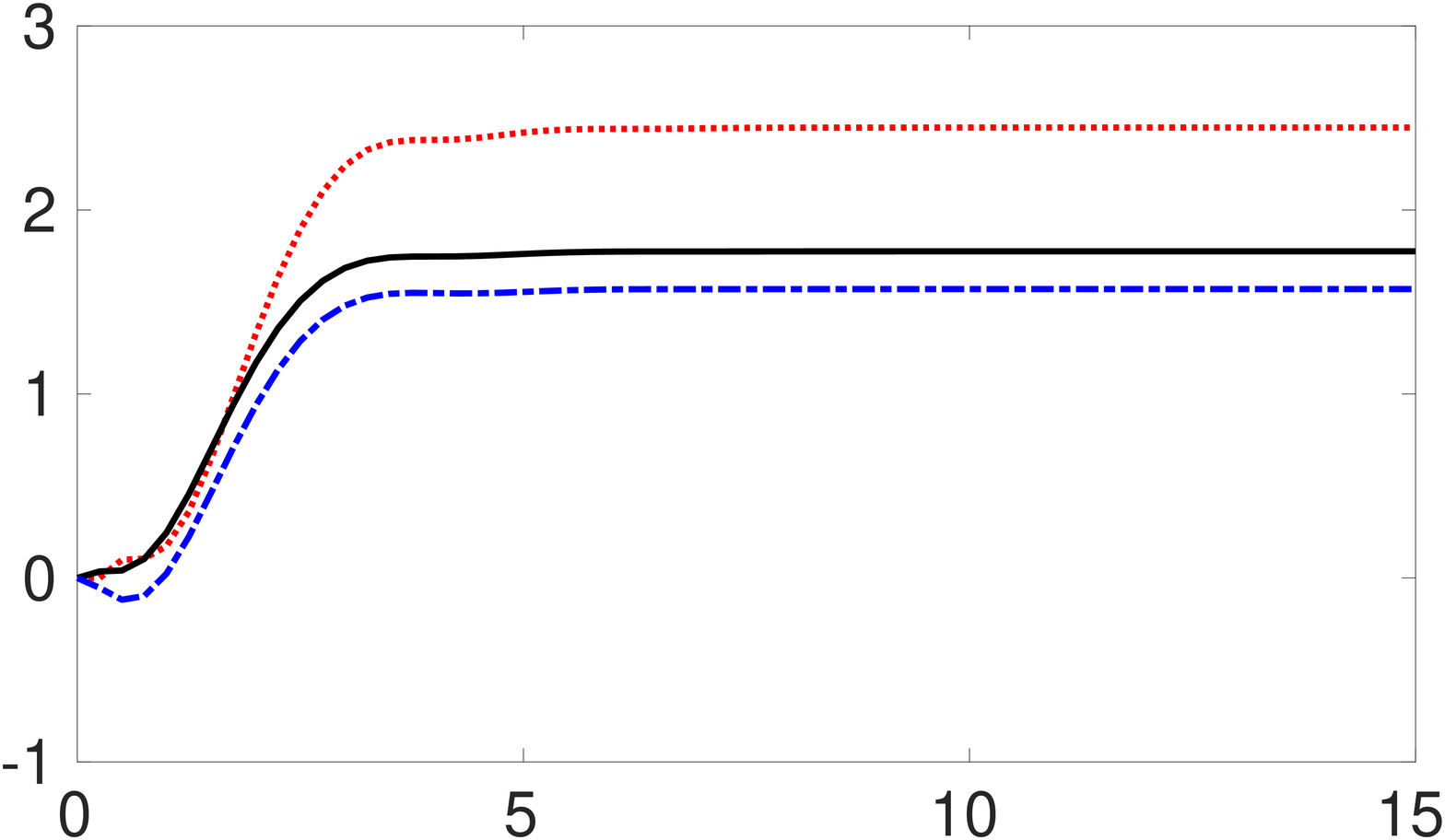} 
\caption{Energy (left) and entropy (right) as function of time for the numerical reference solution (red dotted),
  the minimizing-movements scheme (solid), and the Euler scheme (blue, dash-dotted).}
\label{entropy1}
\end{figure}

A second set of experiments is illustrated in Figure \ref{error}. For
the same choices in \eqref{data} and different time steps 
\begin{equation}\label{times}
\tau^n = 2^{-n} \quad \text{for} \ \ n =-1,0,\dots,11,
\end{equation}
we
compute the uniform errors of the temperature component of the solution and of the energy with
respect to the numerical reference solution.

% \begin{figure}[ht]
% \centering
% \pgfdeclareimage[width=65mm]{error_theta}{figures/error_theta} 
% \pgfdeclareimage[width=65mm]{error_energy}{figures/error_energy} 
% \pgfuseimage{error_theta}\pgfuseimage{error_energy}
% \caption{ Errors with respect to
%   $1/\tau^n$ from \eqref{times}, logarithmic scales for the variable $\theta$ (on the left) and the energy (on the right). The dotted lines indicate order of convergence $1$.
%   Left:  $\max_{t\in [0,T]}|\theta(t)-\haz \theta (t)|$
%   (solid), $\max_{t\in [0,T]}|\theta (t)-\haz \theta^{\,e}(t)|$ (dash-dotted); right: $\max_{t\in [0,T]}| E(\haz
%   y(t)) - E(y^0)|$ (solid), and  $\max_{t\in [0,T]}| E(\haz
%   y^{\,e}(t))  - E(y^0)|$ (dash-dotted). }
% % \caption{Errors with respect to
% %   $1/\tau^n$ from \eqref{times}, logarithmic scales:  $\max_{t\in [0,T]}|\theta(t)-\haz \theta (t)|$
% %   (left, solid),
% %   $\max_{t\in [0,T]}|\theta (t)-\haz \theta^{\,e}(t)|$ (left, dash-dotted), $\max_{t\in [0,T]}| E(\haz
% %   y(t)) - E(y^0)|$ (right, solid), and
% %   $\max_{t\in [0,T]}| E(\haz
% %   y^{\,e}(t))  - E(y^0)|$ (right, dash-dotted). The dotted lines indicate order of convergence $1$.}
% \label{error}
% \end{figure}

\begin{figure}[ht]
\centering
\includegraphics[width=60mm]{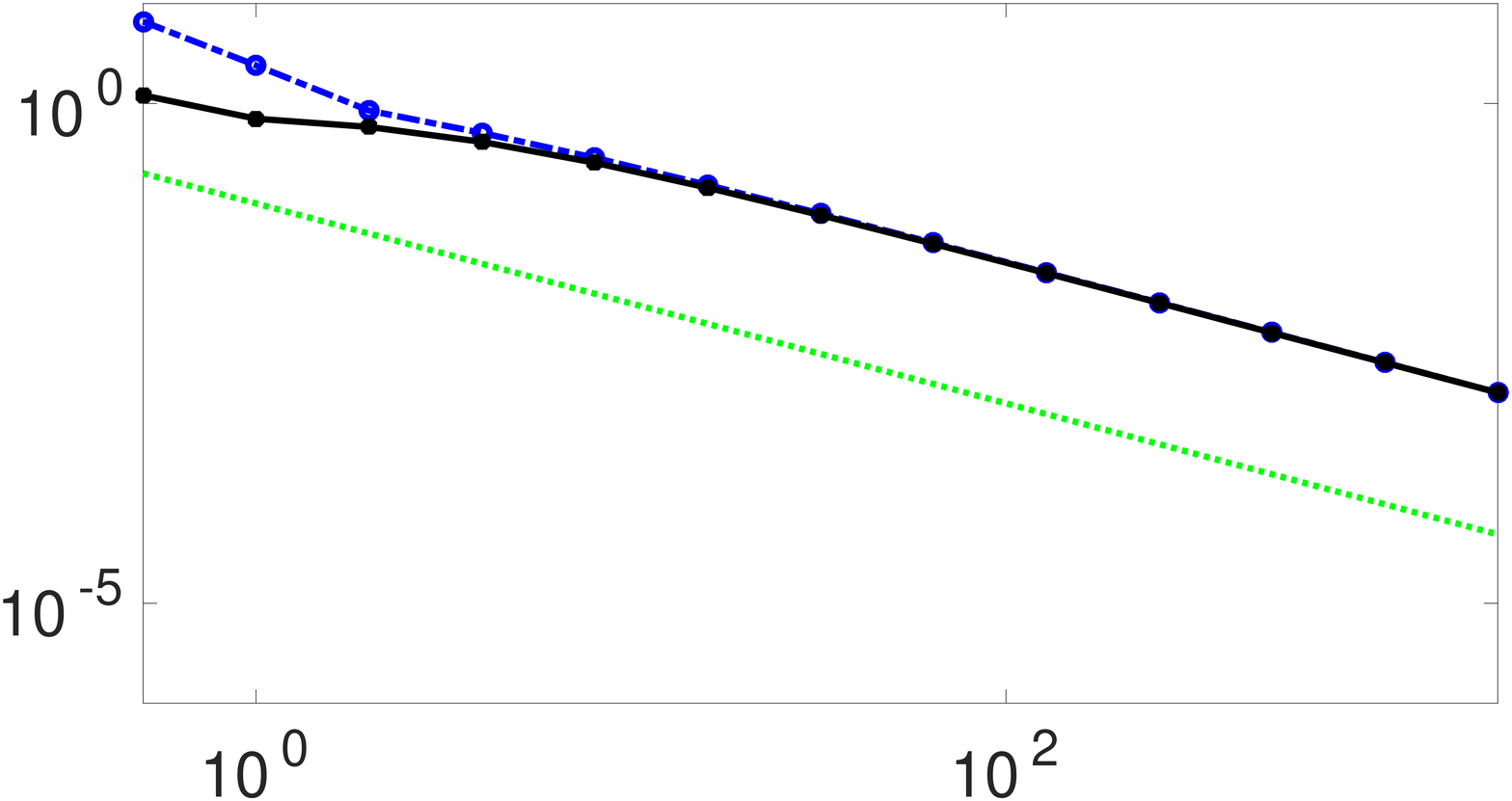} 
\includegraphics[width=60mm]{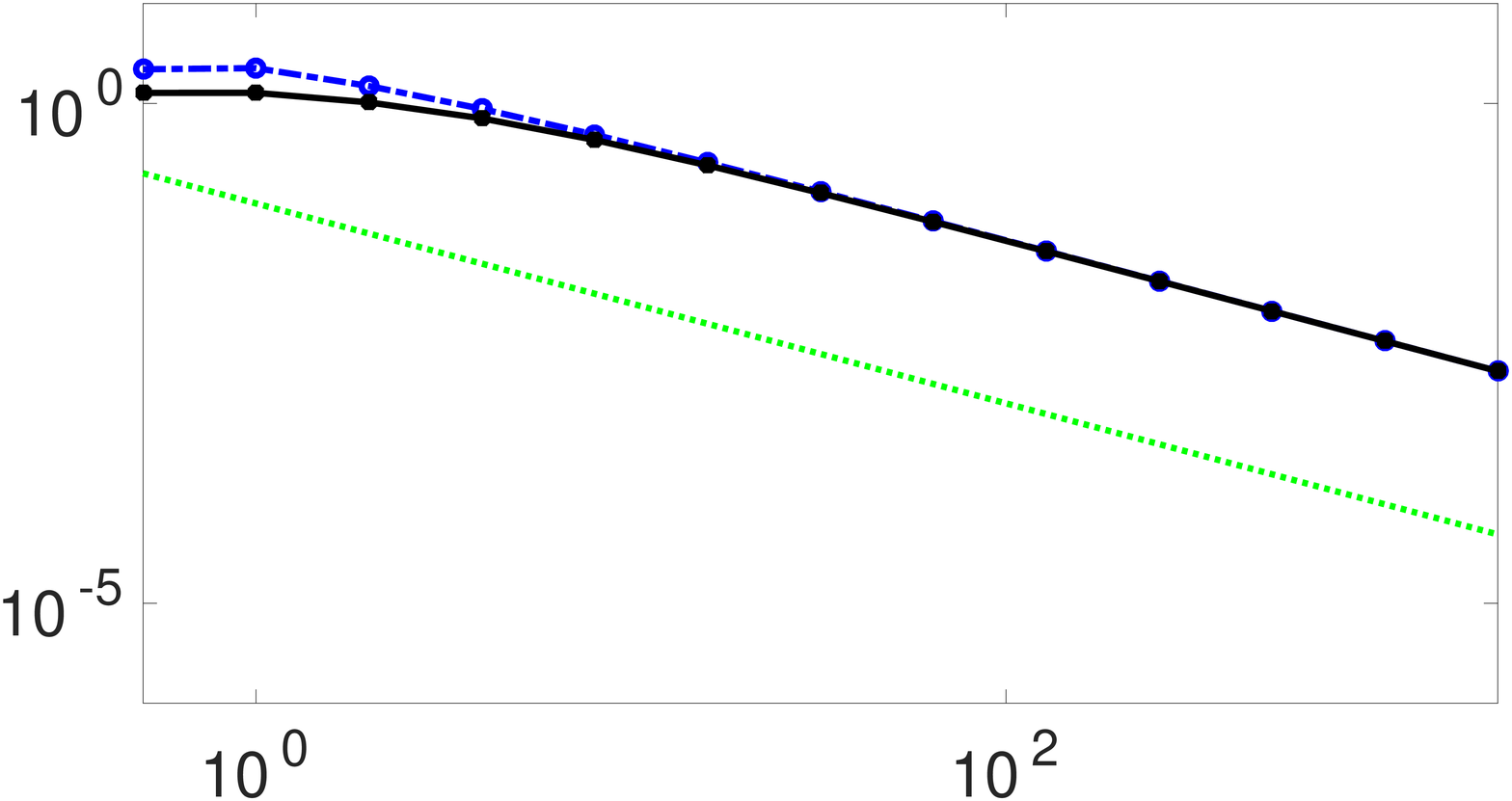} 
\caption{ Errors with respect to
  $1/\tau^n$ from \eqref{times}, logarithmic scales for the variable
  $\theta$ (on the left) and the energy (on the right). The  (green) dotted lines indicate the order of convergence $1$.
  Left:  $\max_{t\in [0,T]}|\theta(t)-\haz \theta (t)|$
  (solid), $\max_{t\in [0,T]}|\theta (t)-\haz \theta^{\,e}(t)|$ (blue, dash-dotted); right: $\max_{t\in [0,T]}| E(\haz
  y(t)) - E(y^0)|$ (solid), and  $\max_{t\in [0,T]}| E(\haz
  y^{\,e}(t))  - E(y^0)|$ (blue, dash-dotted). }
% \caption{Errors with respect to
%   $1/\tau^n$ from \eqref{times}, logarithmic scales:  $\max_{t\in [0,T]}|\theta(t)-\haz \theta (t)|$
%   (left, solid),
%   $\max_{t\in [0,T]}|\theta (t)-\haz \theta^{\,e}(t)|$ (left, dash-dotted), $\max_{t\in [0,T]}| E(\haz
%   y(t)) - E(y^0)|$ (right, solid), and
%   $\max_{t\in [0,T]}| E(\haz
%   y^{\,e}(t))  - E(y^0)|$ (right, dash-dotted). The dotted lines indicate order of convergence $1$.}
\label{error}
\end{figure}

As $\tau$ converges to $0$, our computations confirm that both the minimizing-move\-ments and the Euler
scheme are of order $\tau$, see also \eqref{errore}. Let us mention
that a proof of first-order convergence for the minimizing-movements scheme in the nondissipative
regime ($L=0$, $K$ independent of the state) is given in \cite[Prop.\ 4.3]{generic_euler}. Both schemes do not conserve
energy. Still, the minimizing-movements scheme is more accurate than
the Euler scheme when the time steps are large.

%%%%%%%%%%%%%%%%%%%%%%%%%%%%%%%%%%%%%%%%%%%%%%%%%%%%%%%%%%%%%%%%%%%%%%%%%%%%%%%%%

\section*{Acknowledgements}  
This research is supported by Austrian Science Fund (FWF) project
F\,65 and W\,1245. 
AJ is supported by the FWF projects P\,30000 and P\,33010. 
US is supported by the FWF projects I\,4354 and P\,32788 and by the 
Vienna Science and Technology Fund (WWTF) through Project
MA14-009.  

%%%%%%%%%%%%%%%%%%%%%%%%%%%%%%%%%%%%%%%%%%%%%%%%%%%%%%%%%%%%%%%%%%%%%%%%%%%%%%%%%

\end{document}